\documentclass[11pt]{amsart}
\pdfoutput=1
\usepackage{amsthm,amssymb,bbm,color,enumerate}
\usepackage{graphicx}
\usepackage[UKenglish]{babel}
\usepackage{lmodern}
\usepackage{microtype}
\usepackage{comment}
\usepackage{hyperref}

\theoremstyle{plain}
\newtheorem{theorem}{Theorem}
\newtheorem{thm}[theorem]{Theorem}
\newtheorem{lemma}[theorem]{Lemma}
\newtheorem{corollary}[theorem]{Corollary}

\theoremstyle{definition}
\newtheorem{definition}[theorem]{Definition}

\theoremstyle{remark}
\newtheorem{remark}[theorem]{Remark}

\numberwithin{theorem}{section}
\numberwithin{equation}{section}

\renewcommand{\div}{\ensuremath{\mathrm{div}}}

\newcommand*{\R}{\mathbb{R}}

\newcommand*{\Z}{\mathbb{Z}}

\newcommand*{\eps}{\varepsilon}
\newcommand*{\doo}{\partial}
\newcommand*{\ol}[1]{\overline{#1}}
\newcommand{\sulut}[1]{\left( #1 \right)}
\newcommand{\joukko}[1]{\left\{ #1 \right\}}
\newcommand{\abs}[1]{\left\lvert #1 \right\rvert}

\newcommand{\norm}[1]{\left\| #1 \right\|}
\newcommand{\der}{\mathrm{d}}

\newcommand{\ip}[2]{\left\langle#1,#2\right\rangle}

\DeclareMathOperator{\dist}{dist}

\DeclareMathOperator{\cessco}{conv}

\DeclareMathOperator*{\essinf}{ess \, inf}
\DeclareMathOperator*{\esssupp}{supp}

\newcommand{\kommentar}[1]{}

\title[Monotonicity and enclosure methods for the $p$-Laplacian]{Monotonicity and enclosure methods for the $p$-Laplace equation}

\author{Tommi Brander}
\address{University of Jyv\"askyl\"a, Department of Mathematics and Statistics}
\email{tommi.o.brander@jyu.fi}
\author{Bastian Harrach}
\address{Department of Mathematics,
Goethe University Frankfurt, Germany}
\email{harrach@math.uni-frankfurt.de}
\author{Manas Kar}
\address{NCTS-National Center for Theoretical Sciences,
National Taiwan University, Taiwan}
\email{manas.kar@ncts.ntu.edu.tw}
\author{Mikko Salo}
\address{University of Jyv\"askyl\"a, Department of Mathematics and Statistics}
\email{mikko.j.salo@jyu.fi}

\begin{document}

\begin{abstract} 
We show that the convex hull of a monotone perturbation of a homogeneous background conductivity in the $p$-conductivity equation is  determined by knowledge of the nonlinear Dirichlet-Neumann operator.
We give two independent proofs, one of which is based on the monotonicity method and the other on the enclosure method.
Our results are constructive and require no jump or smoothness properties on the conductivity perturbation or its support.
\end{abstract}

\maketitle





\section{Introduction}

We consider the shape reconstruction problem for nonlinear partial differential equations of $p$-Laplace type.
More precisely, we are interested in identifying the shape and location of an unknown inclusion from boundary voltage to current measurements. Assume $\Omega\subset\mathbb{R}^n, n\geq 2$, to be a bounded open set with conductivity $\sigma\in L_{+}^{\infty}\sulut{\Omega}$, where 
\[
L_{+}^{\infty}\sulut{\Omega} = \joukko{f \in L^\infty\sulut{\Omega}; \essinf f > 0}.
\]
We assume that the conductivity is constant (taken to be $1$ for simplicity) outside an unknown inclusion $D$, so that the inclusion is the set 
\[
D = \mathrm{supp}(\sigma-1).
\]
We study the problem of detecting inclusions from boundary measurements for the $p$-conductivity equation 
\begin{equation} \label{eq:p-conductivity} 
\begin{cases}
\div(\sigma(x)\abs{\nabla u}^{p-2}\nabla u) = 0 \ \text{in}\ \Omega \\
u = f \ \text{on}\ \doo\Omega,
\end{cases}
\end{equation}
where the exponent $p$ is in the range $1<p<\infty$. Given a Dirichlet boundary condition $f\in W^{1,p}(\Omega)/W_{0}^{1,p}(\Omega)$, the forward problem \eqref{eq:p-conductivity} is well-posed in $W^{1,p}(\Omega)$ and the weak solution minimizes the energy functional
\[
E_{\sigma}(v) := \int_{\Omega}\sigma \abs{\nabla v}^p dx
\]  
over all $v\in W^{1,p}(\Omega)$ with $v-f \in W_{0}^{1,p}(\Omega)$ (see e.g.\ \cite{Salo:Zhong:2012, Brander:Ilmavirta:Kar:2015}).

The boundary voltage to current map, also called the nonlinear Dirichlet-Neumann~(DN) map, is the map 
\[
\Lambda_{\sigma} : X \rightarrow X'
\]
defined formally by
\[
\Lambda_{\sigma}(f) := \sigma \abs{\nabla u}^{p-2}\nabla u \cdot \nu|_{\doo\Omega},
\]
where $X := W^{1,p}(\Omega)/W_{0}^{1,p}(\Omega)$ and $X'$ is the dual of $X$, and $\nu$ is the unit outward normal to $\doo\Omega$.
See section~\ref{subsec:DN} for the precise weak definition of $\Lambda_{\sigma}$.

In the special case $p=2$, the equation~\eqref{eq:p-conductivity} is the well-known conductivity equation appearing in Calder\'on's inverse problem~\cite{Calderon:1980}. This problem has been studied extensively in the last 35~years, see the survey \cite{Uhlmann:2014} for recent results. Our problem is an analogue of the standard Calder\'on problem for nonlinear $p$-Laplace type equations. These appear as models in various physical phenomena, e.g.\ nonlinear dielectrics, plastic moulding, electro-rheological and thermo-rheological fluids, fluids governed by a power law, viscous flows in glaciology, or plasticity. Also, the $0$- and $1$-Laplacians have applications in ultrasound modulated electrical impedance tomography (UMEIT) and current density imaging (CDI). See the references in \cite{Brander:Kar:Salo:2015} and \cite[section 2.1]{Brander:2016:apr} for more details.

Our purpose is to detect the shape and location of the inclusion $D$ from boundary measurements for the $p$-conductivity equation~\eqref{eq:p-conductivity}, as encoded by the DN map $\Lambda_{\sigma}$. More precisely, we will be able to reconstruct the essential convex hull of $D$ assuming that $\sigma-1$ does not change sign. Such a result was proved in \cite{Brander:Kar:Salo:2015} for inclusions having some regularity and sharp jumps at the interface, by extending the enclosure method of Ikehata~\cite{Ikehata:1998,Ikehata:1999:jan} to this nonlinear model. In this paper we 
remove all regularity and interface jump assumptions, and also show that monotonicity based shape reconstruction methods~\cite{Tamburrino:Rubinacci:2002,
Tamburrino:2006,Harrach:Ullrich:2013} work in the nonlinear case and allow us to find the convex hull of the inclusion. 

The following theorem is the main result of this paper. It does not require any regularity or jump properties for the inclusion, and we obtain this result using both the monotonicity and the enclosure method. See Section~\ref{sec:prelim} for the definition of the essential convex hull.

\begin{theorem} \label{cor:main}
Suppose~$\Omega \subset \R^n$ is a bounded and open set, $n \geq 2$, and $1 < p < \infty$.
Consider a conductivity~$\sigma \in L^\infty_+(\Omega)$ such that either $\sigma \geq 1$ almost everywhere or $\sigma \leq 1$ almost everywhere.
Then we can recover the essential convex hull of $\mathrm{supp}(\sigma-1)$ from the DN map $\Lambda_{\sigma}$, and we can also determine whether $\sigma$ is less than or greater than 1 almost everywhere.
\end{theorem}

In addition, we give an alternative proof for the boundary determination result of Salo and Zhong~\cite[Theorem 1.1]{Salo:Zhong:2012}. We use an enclosure type method to recover a continuous conductivity on the boundary of a strictly convex domain, whereas \cite{Salo:Zhong:2012} considers domains with~$C^1$ boundary with no convexity assumptions.

\begin{theorem}
Let $\Omega \subset \R^n$ be bounded, open and strictly convex. Let $\sigma_1$, $\sigma_2$ be positive continuous functions on~$\ol\Omega$. If $\Lambda_{\sigma_1} = \Lambda_{\sigma_2}$, then $\sigma_1|_{\doo \Omega} = \sigma_2|_{\doo \Omega}$.
\end{theorem}

We now give some more details and introductory remarks on the monotonicity and the enclosure
method. The monotonicity method is based on the monotonicity relation
\begin{equation}\label{Monotonicity}
\sigma_1 \geq \sigma_0 \ \implies \ \Lambda_{\sigma_1}
 \geq \Lambda_{\sigma_0}.
 \end{equation}  
On the left hand side of (\ref{Monotonicity}), $\sigma_1\geq \sigma_0$ is to be understood pointwise almost everywhere.
In the case $p=2$, the DN maps are linear, and $\Lambda_{\sigma_1} \geq \Lambda_{\sigma_0}$ can be understood in the
sense of operator definiteness (the Loewner partial order). Monotonicity relations such as (\ref{Monotonicity}) allow to constructively determine inclusions by choosing a small test set $B$ and a contrast level $\alpha$ and checking whether the DN map for the test conductivity $1+\alpha \chi_B$ is larger or smaller than $\Lambda_\sigma$. This approach has been proposed and numerically tested by Tamburrino and Rubinacci~\cite{Tamburrino:2006,Tamburrino:Rubinacci:2002}. To show that test sets $B$ outside the true inclusion will not give false positive results in the monotonicity method, one requires a non-trivial converse of the implication~\eqref{Monotonicity} which has been shown 
by Harrach and Ullrich \cite{Harrach:Ullrich:2013} using the concept of localized potentials \cite{Gebauer:2008}. 
Monotonicity-based arguments have been used to prove theoretical uniqueness results in \cite{Harrach:2009,Harrach:Seo:2010,Harrach:2012,Arnold:Harrach:2013,Harrach:Ullrich:2015,Harrach:Ullrich:2017},
and several recent works study monotonicity-based reconstruction methods, cf.\, e.g.,  
\cite{Harrach:Lee:Ullrich:2015,Su:Tamburrino:Ventre:Udpa:Udpa:2015,Tamburrino:Su:Lei:Udpa:Udpa:2015,Garde:Staboulis:2016,Harrach:Minh:2016,Maffucci:Vento:Ventre:Tamburrino:2016,Tamburrino:Sua:Ventre:Udpa:Udpa:2016}.

In this work, we utilize that for $p\neq 2$, the monotonicity relation~\eqref{Monotonicity} is still valid when $\Lambda_{\sigma_1} \geq \Lambda_{\sigma_0}$ is understood in the sense of a preorder defined by the associated quadratic forms, see lemma~\ref{lemma:mono} and the beginning of section~\ref{sec:monotonicity}.
Using special Wolff solutions we can then prove a converse of the implication~\eqref{Monotonicity}, showing that the union of measure theoretic interiors of balls $B\in \mathcal B$ marked by the monotonicity method generates the essential convex hull of the inclusion,
\[
\cessco(\esssupp(|\sigma-1|))= \cessco\left( \bigcup_{B\in \mathcal B} B^\circ \right).
\]
See theorem~\ref{thm:monotonicity_method} and corollary~\ref{cor:mono}.

The enclosure method uses specific exponential solutions, or complex geometrical optics  (CGO) type solutions, for the conductivity equation as test functions to detect the convex hull of the inclusion~$D$. This method is based on analyzing the behaviour of the indicator function, defined via the DN map and the exponential solutions, to see whether or not the level set of the phase function touches the boundary of the inclusion. The indicator function is given by 
\[
I(t,\tau) = \tau^{-p}\int_{\doo\Omega}(\Lambda_{\sigma}-\Lambda_{1})(f_{\tau}) f_{\tau} \,dS,
\] 
where $t$ and $\tau$ are parameters and $f_{\tau}$ are boundary values of the exponential Wolff solutions given in Section~\ref{Sec:prelim}. Now consider a half-space $H\subset\mathbb{R}^n$ so that the energy of the Wolff solution concentrates in $H$ when the parameter~$\tau$ becomes large. Then we observe the following facts for the indicator function. For large $\tau$, if $H$ does not meet $D$ in a set of positive measure, then
\[
I(t,\tau) \rightarrow 0
\]
and if $H$ meets $D$ in a set of positive measure, then we have 
\[
\abs{I(t,\tau)} \rightarrow \infty.
\]
Many earlier works on the enclosure method consider hyperplanes which barely touch the inclusion~$D$, whereas in this work we consider the situations where, roughly, either~$H \cap D$ has positive measure or $\dist\sulut{H,D}>0$.

We prove our main result, theorem~\ref{cor:main}, with both the monotonicity and the enclosure method. The enclosure method for linear equations allows one to identify various shapes by using CGO solutions with different phase functions. See for instance~\cite{Ikehata:1999:oct, Sini:Yoshida:2012}, which use linear phase functions to determine the convex hull, and \cite{Nakamura:Yoshida:2007} where spherical phase functions are used to reconstruct non-convex parts of the obstacle.
The work~\cite{Nagayasu:Uhlmann:Wang:2011} uses the enclosure method with CGO solutions with polynomial phases, where the energy is concentrated inside a cone, and in this case it is possible to approximate the exact shape of certain types of obstacles.
In~\cite{Dos:Kenig:Salo:Uhlmann:2009}, the authors propose a set of CGO solutions for the linear Schr\"{o}dinger equation with all possible different phases.
So, using all of these different phases, it might be possible to approximate the obstacle up to some obstructions.
Therefore, the enclosure method heavily depends on the CGO solutions and their blow-up properties on a specific region depending on the phase functions we choose. 

On the other hand, monotonicity based shape reconstruction methods depend on constructing solutions which blow up in suitable regions.
Harrach~\cite{Gebauer:2008} was able to produce such solutions for the linear conductivity equation.
These solutions are known as localized potentials, and they can be used to determine the exact shape of the inclusion also in certain cases where the inclusion has indefinite sign.
However, the existence of such solutions is proved by linear functional analysis (this involves a duality argument and the unique continuation principle), and it is not known if localized potentials exist in the present nonlinear model. We will replace the localized potentials by Wolff type solutions for the $p$-Laplace equation. These solutions will have very large energy on one side of a given hyperplane, with very small energy on the other side. We will use this phenomenon to establish a version of the monotonicity method in the nonlinear case.

In the present work we implement these two methods to reconstruct the essential convex hull of the inclusion.
This is the first application of monotonicity based shape reconstruction methods to nonlinear equations; see theorem~\ref{thm:monotonicity_method}.
We reach the same conclusion with the enclosure method; see theorem~\ref{thm:enclosure_method}.
In contrast to many previous works, we do not assume any regularity assumptions on the boundary of the inclusion or any jump conditions for the conductivity. 

Unlike for the linear case, comparatively little is known for inverse problems related to the $p$-Laplace equation. The first boundary determination result is due to Salo and Zhong~\cite{Salo:Zhong:2012}, and boundary determination for the normal derivative of the conductivity was shown by Brander~\cite{Brander:2016:jan}.
Recently, under monotonicity assumptions on the conductivity, an interior uniqueness result has been given by Guo, Kar and Salo~\cite{Guo:Kar:Salo:2016}.
This result is not constructive, but it is not restricted to constant background conductivity.
Brander-Kar-Salo~\cite{Brander:Kar:Salo:2015} detect the convex hull of an inclusion $D$ when the conductivity~$\sigma$ satisfies $\sigma = 1$ in $\Omega\setminus\overline{D}$ and $\sigma \geq 1+ \varepsilon>1$ in~$D$ or $\sigma < 1 - \eps$ in~$D$.
In particular, the conductivity $\sigma$ has a jump discontinuity along the interface $\doo D$. The work~\cite{Brander:Ilmavirta:Kar:2015} considers the inclusion detection problem for models with zero or infinite conductivity. For an introduction to the $p$-Calder\'on problem see the thesis of Brander~\cite{Brander:2016:apr}.

This paper is organized as follows. In section~\ref{sec:prelim}, we discuss the essential convex hull and Wolff solutions required for the main results. Section~\ref{sec:monotonicity} introduces the monotonicity based shape reconstruction method for the $p$-Laplace equation.
In section~\ref{sec:enclosure}, we justify the enclosure method under the present assumptions.
Finally, we establish the boundary determination result for strictly convex domains in section~\ref{sec:boundary}.

\subsection*{Acknowledgements}

T.B.\ and M.S.\ were supported by the Academy of Finland (Centre of Excellence in Inverse Problems Research), and M.K.\ and M.S.\ were also supported by an ERC Starting Grant (grant agreement no~307023).

\section{Notations and preliminaries}
\label{Sec:prelim}  \label{sec:prelim}


Throughout this work let $\Omega \subset \R^n$ be open and bounded and suppose that the conductivity satisfies 
$\sigma \in L^\infty_+ \sulut{\Omega;\R} := \joukko{f \in L^\infty(\Omega;\R); \essinf f > 0}$.
We can assume that $\Omega$ is connected without any loss of generality.

\subsection{Essential support and convexity}
\label{subsec:esssuppconvex}

We summarize some basic facts on the (essential) support and its (closed essential) convex hull.
Here and in the following, measure theoretical terms (such as measurable, almost everywhere, or null set) are always used with respect to the $n$-dimensional Lebesgue measure, which we denote by $m$.
For elementary facts concerning convexity and half-spaces see for example~\cite{Rockafellar:1970}.

\begin{definition}[Essential support]
The (essential) support $\esssupp f$ of a measurable function $f:\ \Omega\to \R$ is the complement of the union of all
open sets $O\subseteq \Omega$ where $f|_O=0$ almost everywhere.
\end{definition}

\begin{definition}[Measure theoretic interior]
The measure theoretic interior $A^\circ$ of a measurable set $A$ is the set of all points $x\in A$ with density $1$, i.e. 
\[
\lim_{\epsilon\to 0} \frac{m(A\cap B_\epsilon(x))}{m(B_\epsilon(x))}=1.
\]
\end{definition}

\begin{definition}[Closed half space]\label{def:closed_half_space}
For $\rho\in \R^{n}$ and $t\in \R$, we define the closed half space
\[
H_{\rho,t}:=\left\{ x \in \R^n:\ x \cdot \rho \leq t\right\}.
\]
\end{definition}

\begin{definition}[Closed essential convex hull]\label{def:cessco}
For a measurable set $A \subset \R^n$, the (closed essential) convex hull $\cessco A$ is the intersection of all closed half-spaces $H_{\rho,t}$ ($\rho\in \R^n$, $t\in \R$) such that
\begin{equation}
m\sulut{A \setminus H_{\rho,t}} = 0.
\end{equation}
\end{definition}

\begin{definition}[Essential convex support function] \label{def_essential_support_function}
If $A \subseteq \R^n$ is a bounded measurable set with positive measure, then we define its essential convex support function~$h_A:\ S^{n-1} \to \R$ of~$A$ by 
\begin{equation}
h_A(\rho) = \inf \joukko{t \in \R: \  m\sulut{A \setminus H_{\rho,t}}=0} .
\end{equation}
\end{definition}
Since $A$ has positive measure, the infimum is taken over a non-empty set.
The boundedness of $A$ ensures that the set $\joukko{t \in \R: \  m\sulut{A \setminus H_{\rho,t}}=0}$ is bounded from below.
We have
\begin{align*}
t \ge h_A(\rho) &\implies m(A \setminus \{ x \cdot \rho \leq t \}) = 0, \\
t < h_A(\rho) &\implies m(A \setminus \{ x \cdot \rho \leq t \}) > 0.
\end{align*}

\begin{lemma}\label{lemma:convexhullsupportproperties}
\begin{enumerate}[(a)]
\item For a measurable set $A$, $\cessco A$ is a closed and convex set.
\item If $A$ is measurable, then the Lebesgue density of $A$ is zero for all points $x\in A\setminus \cessco(A)$. In particular,
\[
A^\circ\subseteq \cessco(A) \quad \text{ and } \quad A\subseteq \cessco(A) \cup N \quad \text{ with a null set $N$.}
\]
\item For measurable sets $A,B\subseteq \R^n$ and a null set $N\subseteq \R^n$,
\[
B\subseteq \cessco(A)\cup N \quad \text{ implies } \quad \cessco(B)\subseteq \cessco(A).
\]
\item For a measurable function $f:\ \Omega\to \R$, $\cessco(\esssupp f)$ is
the intersection of all closed halfspaces~$H_{\rho,t}$ with $f|_{\Omega\setminus H_{\rho,t}}=0$ almost everywhere.
\end{enumerate}
\end{lemma}
\begin{proof}
\begin{enumerate}[(a)]
\item
The convex hull~$\cessco A$ is an intersection of closed and convex half-spaces and thus closed and convex.

%
%

\item
Let $x\in A\setminus \cessco (A)$. By definition \ref{def:cessco} there exists a closed half space $H_{\rho,t}$ with $m(A\setminus H_{\rho,t})=0$ and $x\not\in H_{\rho,t}$. Since $H_{\rho,t}$ is closed, $B_\epsilon(x)\cap H_{\rho,t}=\emptyset$ for all sufficiently small balls $B_\epsilon(x)$, so that
\begin{align*}
m(A\cap B_\epsilon(x))&\leq m( (H_{\rho,t}\cap B_\epsilon(x) ) \cup (A\setminus H_{\rho,t}))\\
&\leq m( H_{\rho,t}\cap B_\epsilon(x) ) + m(A\setminus H_{\rho,t}) =0.
\end{align*}
This shows that the Lebesgue density of $A$ is zero in $x$.

In particular, $A^\circ\subseteq \cessco(A)$ and,
by the Lebesgue's density theorem, $m(A\setminus \cessco A)=0$. 
\item
Let $B\subseteq \cessco(A)\cup N$. Then $B$ is a subset of $H_{\rho,t}\cup N$ for every 
half-space $H_{\rho,t}$ with $m\sulut{A \setminus H_{\rho,t}} = 0$.
Hence, $m\sulut{B \setminus H_{\rho,t}} = 0$ for all half-spaces $H_{\rho,t}$ with $m\sulut{A \setminus H_{\rho,t}} = 0$, i.e.,
the set of halfspaces with $m\sulut{B \setminus H_{\rho,t}} = 0$ is a superset of those with $m\sulut{A \setminus H_{\rho,t}} = 0$.
Thus the intersection of all half-spaces with $m\sulut{B \setminus H_{\rho,t}} = 0$ is a subset of
the intersection of all half-spaces with $m\sulut{A \setminus H_{\rho,t}} = 0$, which shows $\cessco(B)\subseteq \cessco(A)$.

\item
By definition, $\cessco(\esssupp(f))$ is the intersection of all closed halfspaces~$H$ for which
$\esssupp(f)\setminus H$ is a null set. Hence, it suffices to show that a closed halfspace $H$ fulfills $m(\esssupp(f)\setminus H)=0$ if and only if $f|_{\Omega\setminus H}=0$ almost everywhere.

Let $O$ be the union of all open sets on which $f$ is zero almost everywhere. Then
\[
\esssupp(f)\setminus H=\left( \Omega\setminus O\right)\setminus H
= \Omega\setminus ( H\cup  O )=\left( \Omega\setminus H\right)\setminus O.
\]
If $m(\esssupp(f)\setminus H)=0$ then $\Omega\setminus H$ is a subset of $O\cup N$ with some null set~$N$,
so $f|_{\Omega\setminus H}=0$ almost everywhere. 
On the other hand, if $f|_{\Omega\setminus H}=0$ a.e.\ then $\Omega\setminus H\subseteq O \cup N$ where $N$ is a null set, and thus
$\left( \Omega\setminus H\right)\setminus O=\esssupp(f)\setminus H$ is a null set.
\end{enumerate}
\end{proof}


\subsection{Wolff solutions and monotonicity for the nonlinear DN map}
\label{subsec:DN}

Let $\Lambda_\sigma$ be the (nonlinear) Dirichlet-Neumann (DN) map for the weighted $p$-Laplace equation ($1<p<\infty$), or $p$-conductivity equation, i.e.\
\[
\Lambda_\sigma:\ X\to X'
\]
defined by 
\[
\left( \Lambda_\sigma(f),g\right) :=\int_\Omega \sigma |\nabla u_\sigma^f|^{p-2} \nabla u_\sigma^f \cdot \nabla v^g \der x, \qquad f, g \in X,
\]
where $X:=W^{1,p}(\Omega)/W_0^{1,p}(\Omega)$, $v^g\in W^{1,p}(\Omega)$ is any representative of the quotient space element $g\in X$, and
$u_\sigma^f$ is the weak solution of the weighted $p$-Laplacian with boundary value $f$, i.e.\ $u_\sigma^f$ is the unique minimizer in $f + W_0^{1,p}(\Omega)$ of the $p$-Dirichlet energy functional
\[
u\mapsto E_\sigma(u):=\int_\Omega \sigma |\nabla u|^p \der x.
\]
The map $\Lambda_\sigma$ is well-defined by this definition~\cite[section 3.2]{Brander:Ilmavirta:Kar:2015}.

We will use a certain exponential solution for the $p$-Laplace equation, which is real valued and periodic in one direction and exponentially behaving in the other direction. This specific solution is the Wolff solution, which was introduced by Wolff~\cite{Wolff:2007} and has applications in inverse problems including boundary determination~\cite{Brander:2016:jan, Salo:Zhong:2012} and inclusion detection~\cite{Brander:Ilmavirta:Kar:2015,Brander:Kar:Salo:2015}.
We explicitly write the solutions as follows (see~\cite[Lemma 3.1]{Brander:Kar:Salo:2015}).

\begin{lemma}[Wolff solutions]
Let $\rho, \rho^{\perp} \in \mathbb{R}^n$ satisfy $\abs{\rho} = \abs{\rho^{\perp}} = 1$ and $\rho \cdot \rho^{\perp} = 0$. Define $h \colon \R^n \to \R$ by $h(x) = e^{-\rho \cdot x}w(\rho^\perp \cdot x)$, where the function~$w$ satisfies the differential equation
\begin{equation}\label{eq:wolff}
w''(s) + V(w,w')w = 0
\end{equation}
with
\begin{equation}
V(w,w') = \frac{(2p-3)\left(w'\right)^2+(p-1)w^2}{(p-1)\left(w'\right)^2 + w^2},
\end{equation}
The function $h$ is then $p$-harmonic.

Given any initial conditions $(a_0,b_0) \in \R^2 \setminus \{(0,0)\}$ there exists a solution~$w \in C^\infty(\R)$ to the differential equation~\eqref{eq:wolff} which is periodic with period $\lambda_p > 0$, satisfies the initial conditions $(w(0),w'(0)) = (a_0,b_0)$ and $\int_0^{\lambda_p} w(s) \,ds = 0$.
Furthermore, there exist constants $c$ and $C$ depending on $a_0$, $b_0$ and $p$ such that for all $s \in \R$ we have
\begin{equation}\label{eq:wolff_new}
C > (w(s))^2+(w'(s))^2 > c > 0.
\end{equation}
\end{lemma}
For a large parameter $\tau \in \R$ and a fixed constant $t \in \R$, we now define the Wolff type solutions $u_\tau \colon \R^n \to \R$ by
\begin{equation}\label{eq:real_test}
u_\tau(x) = e^{\tau(x \cdot \rho - t)} w\left(\tau x \cdot \rho^{\perp}\right).
\end{equation}
The upper and lower bounds for the Wolff solution are due to \cite[lemma~3.1 and equations~(3.5) and (3.6)]{Brander:Kar:Salo:2015}:
\begin{lemma}\label{lemma:Wolff}
Let $\Omega \subset \R^n$ be a bounded open set. There exist $c,C>0$ so that, for each $\rho \in S^{n-1}$, $\tau>0$, and $t \in \R$, the function \eqref{eq:real_test} solves $\mathrm{div}(|\nabla u|^{p-2} \nabla u) = 0$ in $\R^n$, satisfies $u_{\tau}|_{\Omega} \in C^{\infty}(\ol{\Omega}) \subseteq W^{1,p}(\Omega)$, and 
\[
c \tau e^{\tau (x\cdot \rho-t)} \leq |\nabla u_\tau(x)| \leq C \tau e^{\tau (x\cdot \rho-t)}, \quad  x \in \Omega.
\]
\end{lemma}

We will also use the monotonicity inequality~\cite[Lemma 2.1]{Brander:Kar:Salo:2015}.

\begin{lemma}[Monotonicity inequality] \label{lemma:mono}
Let $\sigma_0,\sigma_1\in L_+^\infty(\Omega)$, $1<p<\infty$. For every $f\in W^{1,p}(\Omega)$ and corresponding minimizer
$u_0\in f + W_0^{1,p}(\Omega)$ of the energy~$E_{\sigma_0}$, it holds that
\begin{equation}
\begin{split}
(p-1)\int_\Omega \frac{\sigma_0}{\sigma_1^{1/(p-1)}} \left( \sigma_1^\frac{1}{p-1} - \sigma_0^\frac{1}{p-1}\right) |\nabla u_0|^p \der x\\
\leq \left( (\Lambda_{\sigma_1}-\Lambda_{\sigma_0}) f,f\right) \leq \int_\Omega (\sigma_1-\sigma_0) |\nabla u_0|^p \der x.
\end{split}
\end{equation}
\end{lemma}

\section{Monotonicity method for the $p$-Laplacian}
\label{Sec:monotonicity} \label{sec:monotonicity}

As above let $\sigma\in L^\infty_+(\Omega)$. We make the global monotonicity assumption that
either $\sigma(x)\geq 1$ holds (almost everywhere) in $\Omega$ or 
$\sigma(x)\leq 1$ holds (almost everywhere) in $\Omega$. 

We will show that the essential convex hull of the support of $|\sigma-1|$ is uniquely determined by
the DN~operator~$\Lambda_\sigma$ and that it can be recovered by monotonicity tests.

The monotonicity tests are based on comparing $\Lambda_\sigma$ to certain nonlinear test operators 
in the following sense. For two (possibly nonlinear) operators $\Lambda_1,\Lambda_2: X\to X'$ we write that 
\[
\Lambda_1\geq \Lambda_2\quad \text{ if } \quad \left( (\Lambda_1(f)-\Lambda_2(f)),f\right)\geq 0
\quad \text{ for all $f\in X$.}
\]
Note that this defines a reflexive and transitive relation (a preorder), but not necessarily a partial order since antisymmetry may fail.

Let $\Lambda_0$ denote the DN~operator for the homogeneous conductivity~$\sigma_0=1$.
For a measurable set~$B\subset \Omega$, we also introduce the operator $\Lambda_B:\ X\to X'$ by
\[
\left(\Lambda_B(f),g\right):=\int_B |\nabla u_0^f|^{p-2} \nabla u_0^f \cdot \nabla u_0^g \,\der x, \qquad f, g \in X,
\]
where $u_0^h$ is the unique minimizer in $h + W_0^{1,p}(\Omega)$ of the $p$-Dirichlet energy functional $E_{\sigma_0}$ with $\sigma_0=1$.

\begin{thm}\label{thm:monotonicity_method}
Let $B\subseteq \Omega$ be a measurable set with positive measure. For every constant $\alpha>0$,
\begin{enumerate}[(a)]
\item
$
\sigma|_B\geq 1+ \alpha  \text{ implies } 
\Lambda_0+\tilde \alpha \Lambda_B\leq \Lambda_\sigma,
$
\item 
$
\sigma|_B\leq 1-\alpha  \text{ implies }  
\Lambda_0- \alpha \Lambda_B\geq \Lambda_\sigma,
$
\item 
$
\Lambda_0+\alpha \Lambda_B\leq \Lambda_\sigma  \mbox{ implies }  B\subseteq \cessco(\esssupp(|\sigma-1|)) \text{ up to null sets},
$
\item 
$
\Lambda_0-\alpha \Lambda_B\geq \Lambda_\sigma \mbox{ implies }  B\subseteq \cessco(\esssupp(|\sigma-1|)) \text{ up to null sets},
$
\end{enumerate}
where in (a) $\tilde \alpha:=(p-1)( 1- (1+\alpha)^{-\frac{1}{p-1}} )>0$, and in (c) and (d) the notion ''$B \subseteq A$ up to null sets'' means that $m(B \setminus A) = 0$.
\end{thm}

Before we give the proof of theorem~\ref{thm:monotonicity_method} we show that this implies that 
$\Lambda_\sigma$ uniquely determines the convex hull of the support of $|\sigma-1|$:

\begin{corollary}\label{cor:mono}
Let either $\sigma(x)\geq 1$ hold (almost everywhere) in $\Omega$ or 
$\sigma(x)\leq 1$ hold (almost everywhere) in $\Omega$. Then
\[
\cessco(\esssupp(|\sigma-1|))=\cessco\left( \bigcup_{B\in \mathcal B} B^\circ \right),
\]
where $\mathcal B$ is the family of all measurable sets $B\subseteq \Omega$ for which there exists an $\alpha>0$ such that either $\Lambda_0+\alpha \Lambda_B\leq \Lambda_\sigma$ or $\Lambda_0-\alpha \Lambda_B\geq \Lambda_\sigma$.
\end{corollary}
\begin{proof}
By theorem~\ref{thm:monotonicity_method} (c) and (d), for every $B \in \mathcal{B}$ there exists a null set $N$ with
\begin{equation}
B \subseteq \cessco(\esssupp(|\sigma-1|)) \cup N.
\end{equation}
Lemma~\ref{lemma:convexhullsupportproperties} then implies $B^\circ\subseteq \cessco(B) \subseteq  \cessco(\esssupp(|\sigma-1|))$, so that
\begin{equation}
\bigcup_{B\in \mathcal B}  B^\circ  \subseteq \cessco(\esssupp(|\sigma-1|))
\end{equation}
and again by lemma~\ref{lemma:convexhullsupportproperties}
\begin{equation}
\cessco(\esssupp(|\sigma-1|)) \supseteq \cessco \sulut{\bigcup_{B\in \mathcal B} B^\circ}.
\end{equation}

To show ``$\subseteq$'' we assume that 
\[
\cessco(\esssupp(|\sigma-1|)) \setminus \cessco\left( \bigcup_{B\in \mathcal B} B^\circ \right)\neq \emptyset. 
\]
By definition~\ref{def:cessco}, there would then exist a closed halfspace $H$ with 
\begin{equation}\label{eq:corollary_mon_method_contradict}
m\left( \left( \bigcup_{B\in \mathcal B} B^\circ \right) \setminus H\right)=0\quad \text{ and } \quad \cessco(\esssupp(|\sigma-1|)) \setminus H\neq \emptyset.
\end{equation}
With lemma~\ref{lemma:convexhullsupportproperties} this would imply that every closed halfspace~$\tilde H$ with $\abs{\sigma-1}|_{\Omega\setminus \tilde H}=0$ a.e.\ 
must have an intersection of positive measure with the complement of $H$. Hence $|\sigma-1|$ could not be zero a.e.\ on $\Omega \setminus H$.
But then there would exist a measurable set $B\subseteq \Omega \setminus H$ with $B=B^\circ$, $m(B)=m(B\setminus H)>0$ and a constant $\alpha>0$ with $|\sigma-1|\geq \alpha$ in $B$. Theorem~\ref{thm:monotonicity_method}(a) and (b) would then imply that $B \in \mathcal B$ which would  contradict the first part of (\ref{eq:corollary_mon_method_contradict}).
\end{proof}

\noindent \textit{Proof of Theorem~\ref{thm:monotonicity_method}.}
\begin{enumerate}[(a)]
\item
Note that by our global monotonicity assumption, $\sigma|_B\geq 1+\alpha$ implies that we are in the case that
$\sigma\geq 1$ a.e.\ in $\Omega$. Hence we obtain from the monotonicity lemma \ref{lemma:mono} that
\begin{eqnarray*}
\left( (\Lambda_{\sigma}-\Lambda_0)(f),f\right)
&\geq & (p-1)\int_\Omega \frac{1}{\sigma^{1/(p-1)}} \left( \sigma^\frac{1}{p-1} - 1\right) |\nabla u_0|^p \der x\\
&\geq & (p-1)\int_B  \left( 1- \sigma^{-\frac{1}{p-1}} \right) |\nabla u_0|^p \der x\\
&\geq & (p-1)\left( 1- (1+\alpha)^{-\frac{1}{p-1}} \right)( \Lambda_B(f), f ).
\end{eqnarray*}
\item
If $\sigma|_B\leq 1-\alpha$ then we are in the case that
$\sigma\leq 1$ a.e.\ in $\Omega$ and we obtain from the monotonicity lemma \ref{lemma:mono} that
\begin{eqnarray*}
\left( (\Lambda_{\sigma}-\Lambda_0)(f),f\right)
&\leq &\int_\Omega (\sigma-1) |\nabla u_0|^p \der x
\leq \int_B (\sigma-1) |\nabla u_0|^p \der x\\
&\leq &- \alpha ( \Lambda_B(f), f ).
\end{eqnarray*}
\item
We set $D:=\cessco(\esssupp(|\sigma-1|))$ and assume that $B\setminus D$ is not a null set. 
We will prove that 
\begin{equation}\label{Lambda_not_inequality}
\Lambda_0+\alpha \Lambda_B\not \leq \Lambda_\sigma.
\end{equation}

Since shrinking $B$ will lead to a smaller $\Lambda_B$, it suffices to prove (\ref{Lambda_not_inequality}) with $B$ replaced by a 
subset of $B$. Hence, by replacing $B$ with $B\setminus D$, we can assume, without any loss of generality, that 
\[
m(B)>0\quad \text{ and } \quad B\cap D=\emptyset. 
\]
Moreover, it follows from elementary measure theory or Lebesgue's density theorem that there exists a point 
$x\in B$ such that $m(B\cap B_r(x))>0$ for all sufficiently small open balls $B_r(x)$. By replacing $B$ with $B\cap B_r(x)$ for a sufficiently small ball, we can therefore assume without losing generality that
\[
m(B)>0\quad \text{ and } \quad \cessco(B)\cap D=\emptyset. 
\]
Using the Hahn-Banach (or hyperplane) separation theorem (see for example~\cite[corollary 11.4.2]{Rockafellar:1970}) we then obtain a vector $\rho\in \R^n$, $|\rho|=1$, and numbers $t\in \R$, $\epsilon>0$ such that
\begin{alignat*}{2}
x\cdot \rho &< t \quad && \mbox{ for all } x\in D, \quad \mbox{ and }\\
x\cdot \rho &> t+\epsilon \quad && \mbox{ for all } x\in B.
\end{alignat*}
With the Wolff solutions from lemma~\ref{lemma:Wolff}, it follows that there exist constants $c,C>0$ so that for each $\tau>0$
there exists a solution $u_{0,\tau}$ of the homogeneous $p$-Laplace equation with 
\[
c \tau e^{\tau (x\cdot \rho-t)} \leq |\nabla u_{0,\tau}(x)| \leq C \tau e^{\tau (x\cdot \rho-t)} \quad \forall x\in \Omega.
\]
With $f_\tau:=u_{0,\tau}|_{\doo \Omega}$ it follows that
\begin{eqnarray*}
\left( \alpha \Lambda_B(f_\tau),f_\tau \right)= \alpha \int_B |\nabla u_{0,\tau}|^p \der x
\geq \alpha m(B) c^p \tau^p e^{p \tau \epsilon},
\end{eqnarray*}
and using the monotonicity lemma~\ref{lemma:mono} we obtain
\begin{equation}
\begin{split}
\left( (\Lambda_{\sigma}-\Lambda_0)(f_\tau),f_\tau \right) 
&\leq  \int_\Omega (\sigma-1) |\nabla u_{0,\tau}|^p \der x \\
&\leq \norm{\sigma - 1}_{L^\infty(\Omega)} C^p \tau^p m(D).
\end{split}
\end{equation}
For large enough $\tau$ we have that
\begin{equation}
\norm{\sigma - 1}_{L^\infty(\Omega)} C^p m(D)<\alpha m(B) c^p  e^{p \tau \epsilon}
\end{equation}
and thus 
\[
\left( \alpha \Lambda_B(f_\tau),f_\tau \right) > \left( (\Lambda_{\sigma}-\Lambda_0)(f_\tau),f_\tau \right),
\]
which proves (\ref{Lambda_not_inequality}).
\item  As in (c) we set $D:=\cessco(\esssupp(|\sigma-1|))$ and assume w.l.o.g. that $m(B)>0$ and $\cessco(B)\cap D=\emptyset$.
With the same Wolff solutions as in (c) we obtain from the monotonicity lemma \ref{lemma:mono} that
\begin{eqnarray*}
\left( (\Lambda_{\sigma}-\Lambda_0)(f_\tau),f_\tau \right)&\geq & 
(p-1)\int_\Omega  \left( 1 - \sigma^{-\frac{1}{p-1}} \right) |\nabla u_{0,\tau}|^p \der x\\
&\geq & - (p-1) \norm{\sigma^{-1}}_{L^\infty(\Omega)}^{\frac{1}{p-1}} \int_D   |\nabla u_{0,\tau}|^p \der x\\ %
&\geq & -(p-1) \norm{\sigma^{-1}}_{L^\infty(\Omega)}^{\frac{1}{p-1}} m(D) C^p\tau^p,\\
\end{eqnarray*}
so that for large enough $\tau>0$
\[
\left( -\alpha \Lambda_B(f_\tau),f_\tau \right) < \left( (\Lambda_{\sigma}-\Lambda_0)(f_\tau),f_\tau \right),
\]
which shows that $\Lambda_{\sigma}-\Lambda_0\not\leq -\alpha \Lambda_B$. \hfill $\Box$
\end{enumerate}

\section{Enclosure method} \label{sec:enclosure}

We define the indicator function $I_{\rho}(t,\tau)$ as
\begin{equation*}
I_{\rho}(t,\tau) = \tau^{-p} \ip{\sulut{\Lambda_\sigma-\Lambda_1}(f_\tau)}{f_\tau},
\end{equation*}
where $\rho \in \R^n$ is a unit vector, $t \in \R$, $\tau > 0$, and $f_\tau = u_{\tau}|_{\partial \Omega}$ where $u_{\tau}$ are the Wolff solutions given by \eqref{eq:real_test}.

The indicator function for any $\rho$, $t$, and $\tau$ is determined by $\Lambda_{\sigma}$. Thus the next theorem implies that $\Lambda_{\sigma}$ determines the essential convex hull of the inclusion~$D$, where we write 
\[
D = \mathrm{supp}(\sigma-1).
\]

\begin{theorem} \label{thm:enclosure_method}
Suppose $\sigma \in L^\infty_+\sulut{\Omega}$ and either $\sigma \geq 1$ or $\sigma \leq 1$ almost everywhere.
Then
\begin{equation*}
\cessco D = \bigcap_{\rho \in \R^n; \abs{\rho} = 1} \ol{H}_\rho,
\end{equation*}
where
\begin{equation*}
H_\rho = \joukko{ x \in \R^n \,;\, \abs{I_{\rho}(x \cdot \rho,\tau)} \to \infty \text{ as } \tau \to \infty}.
\end{equation*}
\end{theorem}

Before the proof, we give two simple lemmas.

\begin{lemma}
\label{lemma:expdown}
Let $A \subset \R^n$ be a bounded measurable set. Then 
\begin{equation*}
\int_{A \cap \{ x \cdot \rho \leq t \}} e^{p\tau(x \cdot \rho - t)} \der x \leq \frac{C_A}{p\tau}.
\end{equation*}
\end{lemma}
\begin{proof}
If $R > 0$ is such that $A \subset \overline{B(0,R)}$, we have 
\[
\int_{A \cap \{ x \cdot \rho \leq t \}} e^{p\tau(x \cdot \rho - t)} \der x \leq (2R)^{n-1} \int_{-\infty}^t e^{p\tau(s-t)} \der s = \frac{(2R)^{n-1}}{p\tau}. \qedhere
\]
\end{proof}

Before the next lemma recall that $h_D$ is the convex support function of the set~$D$, as defined in section~\ref{sec:prelim}.

\begin{lemma}
\label{lemma:good_set}
Suppose $\sigma \in L^\infty_+\sulut{\Omega}$ and either $\sigma \geq 1$ or $\sigma \leq 1$ almost everywhere. Let also $t < h_D(\rho)$. Then there exists a set~$S \subset D$ which satisfies the following conditions:
\begin{enumerate}
\item $m(S) > 0$.
\item There is $\eps_1 > 0$ such that for all $x\in S$ we have $x \cdot \rho > t + \eps_1$.
\item There is $\eps_2 > 0$ such that for all $x\in S$ we have $\sigma (x) > 1 + \eps_2$ or $\sigma (x) + \eps_2 < 1 $.
\end{enumerate}
\end{lemma}
\begin{proof}
Choose $\eps_1$ so that $0 < \eps_1 < h_D(\rho)-t$, and define 
\[
\tilde{S} = \{ x \in D \,;\, x \cdot \rho > t + \eps_1 \}.
\]
Then (1) and (2) are true for $\tilde{S}$ by the definition of the essential convex hull and~$h_D$.
To also satisfy (3), we observe that $\tilde S$ is a set of positive measure and $\sigma > 1$ (or $\sigma < 1$) in $\tilde S$.
Since
\begin{equation*}
\tilde S = \bigcup_{k \in \Z_+} \joukko{x \in \tilde S ; \abs{\sigma -1} > 1/k},
\end{equation*}
and $\tilde S$ has positive measure, at least one of the sets~$ \joukko{x \in \tilde S ; \abs{\sigma -1} > 1/k}$ must have positive measure. We choose it as the set~$S$.
\end{proof}

\begin{proof}[Proof of Theorem \ref{thm:enclosure_method}]
First, if $\sigma = 1$ almost everywhere, the indicator function vanishes, and the claim is true.
Let us suppose this is not the case. Then $D = \mathrm{supp}(\sigma-1)$ is nonempty, and we consider a fixed direction~$\rho \in \R^{n}$, $\abs{\rho} = 1$.

Suppose first that $\sigma \geq 1$ almost everywhere. If $t > h_D(\rho)$, then we have by the monotonicity inequality (Lemma~\ref{lemma:mono}) and Lemma~\ref{lemma:Wolff} that 
\begin{equation*}
\begin{split}
0 \leq I_{\rho}(t,\tau) &= \tau^{-p } \ip{\sulut{\Lambda_\sigma-\Lambda_1}(f_\tau)}{f_\tau} \leq \tau^{-p}\int_\Omega \sulut{\sigma-1} \abs{\nabla u_\tau}^p \der x \\
&= \tau^{-p}\int_D \sulut{\sigma-1} \abs{\nabla u_\tau}^p \der x \\
&\leq C \int_D (\sigma-1) e^{p\tau(x \cdot \rho - t)} \der x \\
&\leq C \norm{\sigma-1}_{L^{\infty}(\Omega)} \int_{D \cap \{ x \cdot \rho \leq t \}} e^{p\tau(x \cdot \rho - t)} \der x.
\end{split}
\end{equation*}

In the last inequality, we used the fact that $m(D \setminus \{ x \cdot \rho \leq t) \}) = 0$ since $t > h_D(\rho)$. Now, Lemma \ref{lemma:expdown} implies that for any $t > h_D(\rho)$ one has $I_{\rho}(t,\tau) \to 0$ as $\tau \to \infty$.

Next, suppose $t < h_D(\rho)$. By the monotonicity inequality, Lemma~\ref{lemma:mono}, we get
\begin{equation*}
\begin{split}
I_{\rho}(t,\tau) &= \tau^{-p} \ip{\sulut{\Lambda_\sigma-\Lambda_1}(f_\tau)}{f_\tau} \\
&\geq (p-1)\tau^{-p}\int_{D} \sulut{1-\sigma^{-1/(p-1)}} |\nabla u_\tau|^p \der x \\
&\geq (p-1) \int_D \sulut{1-\sigma^{-1/(p-1)}} e^{p\tau(x \cdot \rho - t)} \der x.
\end{split}
\end{equation*}
By Lemma \ref{lemma:good_set}, there is a set $S \subseteq D$ with positive measure so that $x \cdot \rho > t + \eps_1$ and $1-\sigma^{-1/(p-1)} > \eps_3$ on $S$ for some $\eps_1, \eps_3 > 0$. Thus 
\[
\int_D \sulut{1-\sigma^{-1/(p-1)}} e^{p\tau(x \cdot \rho - t)} \der x \geq \int_S \eps_3 e^{p\tau \eps_1} \der x.
\]
This shows that $I_{\rho}(t,\tau) \to \infty \text{ as } \tau \to \infty$.

If instead $\sigma \leq 1$ almost everywhere, then the previous proof works with minor changes and gives that $I_{\rho}(t,\tau) \to 0$ for $t > h_D(\rho)$ and $I_{\rho}(t,\tau) \to -\infty$ for $t < h_D(\rho)$. Also note that, by definition the set $H_{\rho}$ has the following form
\[
H_\rho = \joukko{ x \in \R^n \,;\, \abs{I_{\rho}(x \cdot \rho,\tau)} \to \infty \text{ as } \tau \to \infty}.
\]
Therefore, we have established that
\begin{equation*}
\joukko{x \in \R^n; x \cdot \rho < h_D(\rho)} \subseteq H_{\rho} \subseteq \joukko{x \in \R^n; x \cdot \rho \leq h_D(\rho)},
\end{equation*}
so
\begin{equation*}
\ol{H}_\rho = \joukko{x \in \R^n; x \cdot \rho \leq h_D(\rho)}.
\end{equation*}

By the definition of $h_D$, one has $m\sulut{D \setminus \ol{H}_\rho}=0$ for each direction~$\rho$.
This implies, by definition of the essential convex hull,
\begin{equation*}
\cessco D \subseteq \bigcap_{\rho \in \R^n; \abs{\rho} = 1} \ol{H}_\rho.
\end{equation*}
To see the other direction, consider some $x_0 \in \bigcap_{\rho \in \R^n; \abs{\rho} = 1} \ol{H}_\rho$, and an arbitrary closed  hyperplane~$H$ satisfying $m\sulut{D \setminus H} = 0$.
There exists $\rho_0 \in \R^n, \abs{\rho_0} = 1$, and  $t_0 \in \R $ such that
\begin{equation*}
H = \joukko{x \in \R^n; x \cdot \rho_0 \leq t_0}.
\end{equation*}
By definition of $h_D$ we then have $t_0 \geq h_D(\rho_0)$, so that $\ol{H}_{\rho_0} \subseteq H$ and consequently $x_0 \in H$.
\end{proof}

\begin{remark}
The situation where both of $D_\pm$ have positive measure is similar to what can be found in~\cite[section 5]{Brander:Kar:Salo:2015}.
\end{remark}

\begin{remark}
The behaviour of the indicator function when $t = h_D$ is tricky.
By the same proof as in~\cite[lemma 4.6]{Brander:Kar:Salo:2015} we have $I_{\rho}(h_D,\tau) \leq C$ (note that we have a different power of~$\tau$ in the indicator function).
The lower bound uses Lipschitz regularity and jump condition on the boundary of the inclusion, so it seems unlikely to work in the present case.
\end{remark}

\section{Boundary determination} \label{sec:boundary}

In this section we give a boundary determination result for the $p$-Calder{\'o}n problem. This result was proved in~\cite[Theorem 1.1]{Salo:Zhong:2012} for domains with $C^1$ boundary. We assume that the domain is strictly convex instead.

\begin{theorem} \label{thm:boundary}
Let $\Omega \subset \R^n$, $n \geq 2$, be a bounded open set so that $\ol{\Omega}$ is strictly convex. Suppose $\sigma \in C(\ol{\Omega})$ is a positive function. Then the nonlinear DN map $\Lambda_{\sigma}$ determines $\sigma|_{\doo \Omega}$.
\end{theorem}
\begin{proof}
Let $x_0 \in \doo \Omega$ be an arbitrary boundary point. Since $\ol{\Omega}$ is convex, Minkowski's supporting hyperplane theorem (see for example~\cite[theorem 11.6]{Rockafellar:1970}) implies that there is a closed half-space $H$ with 
\[
\ol{\Omega} \subset H, \quad x_0 \in \partial H.
\]
By strict convexity one has $\ol{\Omega} \cap \partial H = \{ x_0 \}$ (for if $\ol{\Omega} \cap \partial H$ would contain another point $y_0$, then the line segment between $x_0$ and $y_0$ would lie in $\ol{\Omega} \cap \partial H$ and thus also in $\partial \Omega$, contradicting strict convexity). 
The half-space $H$ may be written as 
\[
H = \{ x \in \R^n \,;\, \rho \cdot x \leq t_0 \}
\]
for some unit vector $\rho \in \R^n$, which will be fixed from now on, and some $t_0 \in \R$.

Let $\gamma \in \R_+$ be a constant, which we will use as a test conductivity. Define the indicator function 
\[
I_{\gamma}(t,\tau) = \tau^{-p} \ip{(\Lambda_{\sigma}-\Lambda_{\gamma})(f_{\tau})}{f_{\tau}}
\]
where $t \in \R$, $\tau > 0$, and $f_{\tau} = u_{\tau}|_{\partial \Omega}$ is the boundary value of the Wolff solution $u_{\tau}$ (which depends on $\rho$) that solves $\mathrm{div}(\gamma |\nabla u|^{p-2} \nabla u) = 0$ in $\Omega$.

For any fixed $t < t_0$, we are going to show that 
\begin{align}
\gamma < \sigma(x_0) &\implies I_{\gamma}(t,\tau) \to +\infty \text{ as $\tau \to \infty$,} \label{indicator_gamma_first} \\
\gamma > \sigma(x_0) &\implies I_{\gamma}(t,\tau) \to -\infty \text{ as $\tau \to \infty$.} \label{indicator_gamma_second}
\end{align}
Then $\sigma(x_0) = \inf\,\{ \gamma > 0 \,;\, \lim_{\tau \to \infty} I_{\gamma}(t,\tau) = -\infty \}$, which shows that $\sigma(x_0)$ can be determined from $\Lambda_{\sigma}$ thus concluding the proof.

Suppose that $\gamma \neq \sigma(x_0)$. By continuity of $\sigma$, there are $\eps, r > 0$ such that 
\[
|\sigma(x) - \gamma| \geq \eps, \qquad x \in B(x_0,r) \cap \ol{\Omega}.
\]
For $t \in \R$ consider the set 
\[
S_t = \{ x \in \R^n \,;\, \rho \cdot x > t \} \cap \ol{\Omega}.
\]
We claim that there exists $\delta > 0$ so that 
\[
S_{t_0 - \delta} \subset B(x_0,r) \cap \ol{\Omega}.
\]
For if not, then for any $j$ there is $x_j \in \ol{\Omega}$ with $\rho \cdot x_j > t_0 - 1/j$ but $x_j \notin B(x_0,r)$. The set $\ol{\Omega}$ is compact, hence some subsequence of $(x_j)$ converges to some $x_{\infty} \in \ol{\Omega}$ with $\rho \cdot x_{\infty} \geq t_0$. But we saw earlier that $\ol{\Omega} \subset H$ and $\ol{\Omega} \cap \partial H = \{ x_0 \}$ where $H = \{ \rho \cdot x \leq t_0 \}$, showing that $x_{\infty} = x_0$ which is a contradiction.

Now choose $t$ with $t_0 - \delta < t < t_0$. Then the set $S_t$ is nonempty with positive measure, and one has 
\[
|\sigma(x) - \gamma| \geq \eps, \qquad x \in S_t.
\]
Let $s \in \{ +1, -1 \}$ be the sign of $\sigma-\gamma$ in $S_t$. In the monotonicity inequality, Lemma~\ref{lemma:mono}, we write
\begin{align*}
F_+(a,b) &= (p-1) \frac{a}{b^{1/(p-1)}}\sulut{b^{1/(p-1)} - a^{1/(p-1)}} \\
F_-(a,b) &= a-b
\end{align*}
and estimate
\begin{equation*}
\begin{split}
sI_{\gamma}(t,\tau) &= s\tau^{-p} \ip{\sulut{\Lambda_\sigma-\Lambda_\gamma}f_\tau}{f_\tau} \geq s\tau^{-p}\int_\Omega F_s\sulut{\sigma(x),\gamma}\abs{\nabla u_\tau}^p \der x \\
&=s\tau^{-p} \Bigg(  \int_{S_t} F_s\sulut{\sigma(x),\gamma}\abs{\nabla u_\tau}^p \der x + \int_{\Omega \setminus S_t} F_s\sulut{\sigma(x),\gamma}\abs{\nabla u_\tau}^p \der x \Bigg) \\
&\geq s \Bigg( c \int_{S_t} F_s\sulut{\sigma(x),\gamma} e^{p\tau\sulut{x \cdot \rho - t}} \der x 
 - C \int_{\Omega \setminus S_t} F_s\sulut{\sigma(x),\gamma} e^{p\tau\sulut{x \cdot \rho - t}} \der x  \Bigg) \\
 & \geq c \int_{S_t} e^{p\tau\sulut{x \cdot \rho - t}} \der x - C \int_{\Omega \setminus S_t} e^{p\tau\sulut{x \cdot \rho - t}} \der x.
\end{split}
\end{equation*}
If $t < t' < t_0$, then $S_{t'}$ has positive measure and one has 
\[
\int_{S_t} e^{p\tau\sulut{x \cdot \rho - t}} \der x \geq \int_{S_{t'}} e^{p\tau\sulut{x \cdot \rho - t}} \der x \geq  e^{p\tau\sulut{t' - t}} m(S_{t'}) \to \infty
\]
as $\tau \to \infty$, and 
\[
\int_{\Omega \setminus S_t} e^{p\tau\sulut{x \cdot \rho - t}} \der x \leq \int_{\Omega \setminus S_t} \der x \leq C.
\]
Thus $I_{\gamma} \to +\infty$ when $\sigma(x_0) > \gamma$ and $I_{\gamma} \to -\infty$ when the opposite inequality holds. This proves \eqref{indicator_gamma_first}--\eqref{indicator_gamma_second} under the condition $t_0 - \delta < t < t_0$, but the result holds also for $t \leq t_0-\delta$ since 
\[
I_{\gamma}(t,\tau) = I_{\gamma}(t_0-\delta/2,\tau) e^{p\tau\sulut{t_0-\delta/2-t}}
\]
which follows by a short computation.
\end{proof}

\begin{remark}
It is clear that the same proof also gives the following result: if $\Omega \subset \R^n$ is a bounded open set, if $\Omega$ is strictly convex at $x_0 \in \partial \Omega$ in the sense that there is a closed half-space $H$ with 
\[
\ol{\Omega} \subset H, \quad \ol{\Omega} \cap \partial H = \{ x_0 \},
\]
and if $\sigma \in L^{\infty}_+(\Omega)$ is continuous near $x_0$, then the knowledge of $\Lambda_{\sigma}$ determines $\sigma(x_0)$.
\end{remark}

\begin{remark}
The proof gives the following simple algorithm for boundary determination in strictly convex domains:
\begin{enumerate}
\item First take an arbitrary point $x_0 \in \partial\Omega$. Our aim is to find the value of $\sigma$ at $x_0$.
\item
Select a direction $\rho$ such that $\rho \cdot x_0 = \max_{x \in \ol \Omega} \rho \cdot x = t_0$. This is possible by strict convexity of the set~$\Omega$.
\item
Fix a parameter $t$ so that $t<t_0$, which corresponds to taking a hyperplane that intersects $\Omega$.
\item
Choose a small $\gamma_0\in \mathbb{R}_{+}$ and a large $\gamma^0\in \mathbb{R}_{+}$ such that
\[
I_{\gamma_0} \rightarrow +\infty \ \text{and}\ I_{\gamma^0} \rightarrow -\infty \text{ as $\tau \to \infty$}.
\]
\item
Suppose $\gamma_{j-1}$ and $\gamma^{j-1}$ have been determined. Define $\kappa_{j} = \frac{\gamma_{j-1}+\gamma^{j-1}}{2}$ and calculate $I_{\kappa_{j}}$.
\item
If $I_{\kappa_{j}} \rightarrow +\infty$, take $\gamma_{j} = \kappa_j$. If $I_{\kappa_{j}} \rightarrow -\infty$, take $\gamma^{j} = \kappa_j$. Otherwise $\kappa_j = \sigma(x_0)$, in which case we are done.
\item
Repeat steps (5) and (6).
\item
We have
\begin{equation*}
\lim_{j \to \infty} \kappa_j = \lim_{j \to \infty} \gamma_j = \lim_{j \to \infty} \gamma^j = \sigma(x_0).
\end{equation*}
\end{enumerate}
\end{remark}


\bibliography{math}
\bibliographystyle{abbrv}

\end{document}